\documentclass[a4paper,10pt]{article}
\usepackage[english]{babel}

\usepackage{mathrsfs}
\usepackage{amsthm}
\usepackage{amssymb}
\usepackage{amsmath,amsthm,amsfonts}
\usepackage{hhline}
\usepackage{textcomp}
\usepackage{amsmath,amscd}
\usepackage{amsmath}
\usepackage{amsfonts}
\usepackage{amssymb}
\usepackage{graphicx}

\newtheorem{theorem}{Theorem}[section]
\newtheorem{lemma}{Lemma}[section]

\newtheorem{defn}{Definition}[section]
\newtheorem{cor}{Corollary}[section]
\theoremstyle{definition}

\begin{document}
\title{ Some properties of Graph Laplacians of cyclic groups  }
\author{ Dmitriy Goltsov  \footnote {Institute of mathematics of Ukrainian Academy of Sciences. E-mail: adanos@i.ua}}  
\date{}
\maketitle
\begin{abstract}
In this paper we investigate a spectra of the Laplacian matrix of cyclic groups using the properties 
of their characteristic polynomials. We have proved several assertions about the relationship 
between the spectra of different groups.
\end{abstract}

\section{Introduction} 
Let us consider a graph $G$ with the vertex set $V=\{1,....,n\}$ and the edge set $E$.

\begin{defn} The Laplacian matrix of the Graph G is a matrix
$L(G)=(a_{i,j\in V})$, with
$$
a_{i,j}=\begin{cases}-1&\text{if $ij\in E$}\\d(i)&\text{if $i=j$}\\0&\text{otherwise}\end{cases}
$$
where $d(i)=\vert\lbrace e\in E\vert i\in e\rbrace\vert$ is the degree of the vertex $i$. 
\end{defn}

\begin{defn} The Cayley Graph of a discrete group $L$ with a system of 
generators $S$ is the graph whose vertices are the elements of the group $L$ and 
whose edges are determined by the following condition: if $g$ and $s$ belong to $L$
 then there is an edge from $g$ to $f$ if and only if $ f =  g \ast s$ for some 
 $s\in S\bigcup S^{-1}$.
\end{defn}
Let us consider the Cayley graph of the group $Z_{n}$. 
Note that the Laplacian is a nonnegative operator so all eigenvalues are greater or equal to 0.
If $ n=1 $, then the Laplacian of the Cayley graph of this group is $ \begin{pmatrix} 0 \end{pmatrix}$. 
This matrix has only one eigenvalue which is zero.
If $ n=2 $, then the Laplacian of the Cayley graph of $Z_{2}$ is the matrix

$$ \begin{pmatrix}
 1 & -1 \\ 
 -1 & 1 \\ 
 \end{pmatrix}
$$
The eigenvalues of the Laplacian are $ \lambda = 0 $ and $ \lambda = 2 $.
The Laplacian of $Z_{n}, n>3$ is the next matrix:
$$
\underbrace{\begin{pmatrix}
 2 & -1 & 0 & ... & 0 & -1 \\ 
 -1 & 2 & -1 & ... & 0 & 0 \\ 
 0 & -1 & 2 & ... & 0 & 0 \\ 
 0 & 0 & -1 & ... & 0 & 0 \\ 
 \vdots & \vdots & \vdots & \ddots & \vdots & \vdots \\
 0 & 0 & 0 & ... & 2 & -1 \\ 
 -1 & 0 & 0 & ... & -1 & 2 \\  
\end{pmatrix}}_{n}
$$

\bigskip

This matrix is circulant and its eigenvalues are known. In this paper we use another
method instead of the well-known method of Gray (see~\cite{gray}) by using spectrum to investigate 
 the properties of the spectra and the characteristic polynomials of the Laplacians of cyclic groups.
Let find the determinant of the following matrix. Set $ a = 2-\lambda $.

$$A_{n}:=\underbrace{\begin{vmatrix}
 a & -1 & 0 & ... & 0 & -1 \\ 
 -1 & a & -1 & ... & 0 & 0 \\ 
 0 & -1 & a & ... & 0 & 0 \\ 
 0 & 0 & -1 & ... & 0 & 0 \\ 
 \vdots & \vdots & \vdots & \ddots & \vdots & \vdots \\
 0 & 0 & 0 & ... & a & -1 \\ 
 -1 & 0 & 0 & ... & -1 & a \\  
\end{vmatrix}}_{n} 
=
\underbrace{\begin{vmatrix}
 a & -1 & 0 & ... & 0 & 0 \\ 
 -1 & a & -1 & ... & 0 & 0 \\ 
 0 & -1 & a & ... & 0 & 0 \\ 
 0 & 0 & -1 & ... & 0 & 0 \\ 
 \vdots & \vdots & \vdots & \ddots & \vdots & \vdots \\
 0 & 0 & 0 & ... & a & -1 \\ 
 0 & 0 & 0 & ... & -1 & a \\  
\end{vmatrix}}_{n-1} 
+$$
$$\underbrace{\begin{vmatrix}
 -1 & -1 & 0 & ... & 0 & 0 \\ 
 0 & a & -1 & ... & 0 & 0 \\ 
 0 & -1 & a & ... & 0 & 0 \\ 
 0 & 0 & -1 & ... & 0 & 0 \\
 \vdots & \vdots & \vdots & \ddots & \vdots & \vdots \\ 
 0 & 0 & 0 & ... & a & -1 \\ 
 -1 & 0 & 0 & ... & -1 & a \\  
\end{vmatrix}}_{n-1} 
+(-1)^{n}
\underbrace{\begin{vmatrix}
 -1 & a & -1 & ... & 0 & 0 \\ 
 0 & -1 & a & ... & 0 & 0 \\ 
 0 & 0 & a & ... & 0 & 0 \\ 
 0 & 0 & -1 & ... & 0 & 0 \\
 \vdots & \vdots & \vdots & \ddots & \vdots & \vdots \\ 
 0 & 0 & 0 & ... & -1 & a \\ 
 -1 & 0 & 0 & ... & 0 & -1 \\  
\end{vmatrix}}_{n-1} 
$$
$$ L_{n-1}:= \underbrace{\begin{vmatrix}
 a & -1 & 0 & ... & 0 & 0 \\ 
 -1 & a & -1 & ... & 0 & 0 \\ 
 0 & -1 & a & ... & 0 & 0 \\ 
 0 & 0 & -1 & ... & 0 & 0 \\ 
 \vdots & \vdots & \vdots & \ddots & \vdots & \vdots \\
 0 & 0 & 0 & ... & a & -1 \\ 
 0 & 0 & 0 & ... & -1 & a \\  
\end{vmatrix}}_{n-1} 
=
a
\underbrace{\begin{vmatrix}
 a & -1 & 0 & ... & 0 & 0 \\ 
 -1 & a & -1 & ... & 0 & 0 \\ 
 0 & -1 & a & ... & 0 & 0 \\ 
 0 & 0 & -1 & ... & 0 & 0 \\ 
 \vdots & \vdots & \vdots & \ddots & \vdots & \vdots \\
 0 & 0 & 0 & ... & a & -1 \\ 
 0 & 0 & 0 & ... & -1 & a \\  
\end{vmatrix}}_{n-2}
+$$
$$  
\underbrace{\begin{vmatrix}
 -1 & -1 & 0 & ... & 0 & 0 \\ 
 0 & a & -1 & ... & 0 & 0 \\ 
 0 & -1 & a & ... & 0 & 0 \\ 
 0 & 0 & -1 & ... & 0 & 0 \\
 \vdots & \vdots & \vdots & \ddots & \vdots & \vdots \\ 
 0 & 0 & 0 & ... & a & -1 \\ 
 0 & 0 & 0 & ... & -1 & a \\  
\end{vmatrix}}_{n-1}  
= aL_{n-2}-L_{n-3}\eqno (1)$$ 

$$ \underbrace{\begin{vmatrix}
 -1 & -1 & 0 & ... & 0 & 0 \\ 
 0 & a & -1 & ... & 0 & 0 \\ 
 0 & -1 & a & ... & 0 & 0 \\ 
 0 & 0 & -1 & ... & 0 & 0 \\
 \vdots & \vdots & \vdots & \ddots & \vdots & \vdots \\ 
 0 & 0 & 0 & ... & a & -1 \\ 
 -1 & 0 & 0 & ... & -1 & a \\  
\end{vmatrix}}_{n-1} 
= -L_{n-2} +
\underbrace{\begin{vmatrix}
 0 & -1 & 0 & ... & 0 & 0 \\ 
 0 & a & -1 & ... & 0 & 0 \\ 
 0 & -1 & a & ... & 0 & 0 \\ 
 0 & 0 & -1 & ... & 0 & 0 \\
 \vdots & \vdots & \vdots & \ddots & \vdots & \vdots \\ 
 0 & 0 & 0 & ... & a & -1 \\ 
 -1 & 0 & 0 & ... & -1 & a \\  
\end{vmatrix}}_{n-2} = $$
$$ -L_{n-2} +
\underbrace{\begin{vmatrix}
 0 & -1 & 0 & ... & 0 & 0 \\ 
 0 & a & -1 & ... & 0 & 0 \\ 
 0 & -1 & a & ... & 0 & 0 \\ 
 0 & 0 & -1 & ... & 0 & 0 \\
 \vdots & \vdots & \vdots & \ddots & \vdots & \vdots \\ 
 0 & 0 & 0 & ... & a & -1 \\ 
 -1 & 0 & 0 & ... & -1 & a \\  
\end{vmatrix}}_{n-3} = ... = -L_{n-2}+ 
\begin{vmatrix}
 0 & -1 \\ 
 -1 & a \\ 
\end{vmatrix}
=-L_{n-2}-1$$
$$ \underbrace{\begin{vmatrix}
 -1 & a & -1 & ... & 0 & 0 \\ 
 0 & -1 & a & ... & 0 & 0 \\ 
 0 & 0 & a & ... & 0 & 0 \\ 
 0 & 0 & -1 & ... & 0 & 0 \\
 \vdots & \vdots & \vdots & \ddots & \vdots & \vdots \\ 
 0 & 0 & 0 & ... & -1 & a \\ 
 -1 & 0 & 0 & ... & 0 & -1 \\  
\end{vmatrix}}_{n-1}
=
\underbrace{\begin{vmatrix}
 -1 & a & -1 & ... & 0 & 0 \\ 
 0 & -1 & a & ... & 0 & 0 \\ 
 0 & 0 & a & ... & 0 & 0 \\ 
 0 & 0 & -1 & ... & 0 & 0 \\
 \vdots & \vdots & \vdots & \ddots & \vdots & \vdots \\ 
 0 & 0 & 0 & ... & -1 & a \\ 
 0 & 0 & 0 & ... & 0 & -1 \\  
\end{vmatrix}}_{n-2} +(-1)^{n}L_{n-2} $$

So $$ A_{n}=aL_{n-1}-2L_{n-2}-2\eqno (2)$$ .

Let complete the table of coefficients of $ L_{n} $.

\bigskip

$$\begin{tabular}{|c|c|c|c|c|c|c|c|c|c|c|}
\hline • & 1 & $ a $ & $ a^{2} $ & $ a^{3} $ & $ a^{4} $ & $ a^{5} $ & $ a^{6} $ & $ a^{7} $ & $ a^{8} $ & $ a^{9} $ \\ 
\hline $ L_{1} $ & 0 & 1 & 0 & 0 & 0 & 0 & 0 & 0 & 0 & 0 \\ 
\hline $ L_{2} $ & -1 & 0 & 1 & 0 & 0 & 0 & 0 & 0 & 0 & 0 \\ 
\hline $ L_{3} $ & 0 & -2 & 0 & 1 & 0 & 0 & 0 & 0 & 0 & 0 \\ 
\hline $ L_{4} $ & 1 & 0 & -3 & 0 & 1 & 0 & 0 & 0 & 0 & 0 \\ 
\hline $ L_{5} $ & 0 & 3 & 0 & -4 & 0 & 1 & 0 & 0 & 0 & 0 \\ 
\hline $ L_{6} $ & -1 & 0 & 6 & 0 & -5 & 0 & 1 & 0 & 0 & 0 \\ 
\hline $ L_{7} $ & 0 & -4 & 0 & 10 & 0 & -6 & 0 & 1 & 0 & 0 \\ 
\hline $ L_{8} $ & 1 & 0 & -10 & 0 & 15 & 0 & -7 & 0 & 1 & 0 \\ 
\hline $ L_{9} $ & 0 & 5 & 0 & -20 & 0 & 21 & 0 & -8 & 0 & 1 \\ 
\hline 
\end{tabular}$$ 

\bigskip

We can complete the table of coefficients of $ A_{n} $ from (2) and the table of coefficients of $ L_{n} $.

$$\begin{tabular}{|c|c|c|c|c|c|c|c|c|c|c|c|c|}
\hline • & 1 & $ a $ & $ a^{2} $ & $ a^{3} $ & $ a^{4} $ & $ a^{5} $ & $ a^{6} $ & $ a^{7} $ & $ a^{8} $ & $ a^{9} $ & $ a^{10} $ & $ a^{11} $ \\ 
\hline $ A_{3} $ & -2 & -3 & 0 & 1 & 0 & 0 & 0 & 0 & 0 & 0 & 0 & 0\\ 
\hline $ A_{4} $ & 0 & 0 & -4 & 0 & 1 & 0 & 0 & 0 & 0 & 0 & 0 & 0\\ 
\hline $ A_{5} $ & -2 & 5 & 0 & -5 & 0 & 1 & 0 & 0 & 0 & 0 & 0 & 0\\ 
\hline $ A_{6} $ & -4 & 0 & 9 & 0 & -6 & 0 & 1 & 0 & 0 & 0 & 0 & 0\\ 
\hline $ A_{7} $ & -2 & -7 & 0 & 14 & 0 & -7 & 0 & 1 & 0 & 0 & 0 & 0\\ 
\hline $ A_{8} $ & 0 & 0 & -16 & 0 & 20 & 0 & -8 & 0 & 1 & 0 & 0 & 0\\ 
\hline $ A_{9} $ & -2 & 9 & 0 & -30 & 0 & 27 & 0 & -9 & 0 & 1 & 0 & 0\\ 
\hline $ A_{10} $ & -4 & 0 & 25 & 0 & -50 & 0 & 35 & 0 & -10 & 0 & 1 & 0\\ 
\hline $ A_{11} $ & -2 & -11 & 0 & 55 & 0 & -77 & 0 & 44 & 0 & -11 & 0 & 1\\ 
\hline
\end{tabular} $$
\section{Main results} 
\begin{lemma} \label{f}
$\forall n\in N, \forall k\subseteq [1,..,n]: L_{n}=L_{n-k}L_{k}-L_{n-k-1}L_{k-1} $.
\end{lemma}
\begin{proof} $ L_{n}=aL_{n-1}-L_{n-2}=a(aL_{n-2}-L_{n-3})-L_{n-2}=$
 $(a^{2}-1)L_{n-2}-aL_{n-3}=L_{2}L_{n-2}-L_{1}L_{n-3}$.\\
Assume $ L_{n}=L_{k}L_{n-k}-L_{k-1}L_{n-k-1}$. Then $ L_{n}=L_{k}L_{n-k}-L_{k-1}L_{n-k-1}=L_{k}(aL_{n-k-1}-L_{n-k-2})-L_{k-1}L_{n-k-1}=(aL_{k}-L_{k-1})L_{n-k-1}-L_{k}L_{n-k-2}=L_{k+1}L_{n-k-1}-L_{k}L_{n-k-2}$.
\end{proof}
\begin{lemma} \label{g}
$\forall n\in N: L^{2}_{n-1}=L_{n-2}L_{n}+1$.
\end{lemma}
\begin{proof} $ L_{1}=a, L_{2}=a^{2}-1, L_{3}=a^{3}-2a$. 
$ L^{2}_{2}=L_{3}L_{1}+1 $. Assume $ L^{2}_{k-1}=L_{k-2}L_{k}+1$. Then $L^{2}_{k}=L_{k-1}L_{k+1}+1$; 
$ L^{2}_{k-1}=L_{k-2}L_{k}+1 $;
$ L^{2}_{k-1}=L_{k-2}(aL_{k-1}-L_{k-2})+1 $ and
\begin{equation} \label{h}
 L^{2}_{k-1}+L^{2}_{k-2}-1=aL_{k-1}L_{k-2}
\end{equation}
$L^{2}_{k}=L_{k-1}L_{k+1}+1$;
$(aL_{k-1}-L_{k-2})^{2}=L_{k-1}(aL_{k}-L_{k-1})+1$;
$a^{2}L_{k-1}^{2}-2aL_{k-1}L_{k-2}+L_{k-2}^{2}=aL_{k}L_{k-1}-L_{k-1}^{2}+1$;
$a^{2}L_{k-1}^{2}-2aL_{k-1}L_{k-2}=aL_{k}L_{k-1}-L_{k-1}^{2}-L_{k-2}^{2}+1$.
Then by Equation~\eqref{h} $aL_{k-1}^{2}-aL_{k-1}L_{k-2}-aL_{k}L_{k-1}=0$; $aL_{k-1}(aL_{k-1}-L_{k-2})-aL_{k}L_{k-1}=0$;
$aL_{k}L_{k-1}-aL_{k}L_{k-1}=0$. 

\end{proof}
\begin{lemma} \label{e}
If $ \lambda $ is the eigenvalue of Laplacian of $ Z_{n} $, then $ \lambda $ is the eigenvalue of
the Laplacian of $ Z_{2^{k}n}, \forall k\in N $.
\end{lemma}
\begin{proof} $ A_{2n}=aL_{2n-1}-2L_{2n-2}-2$. Then by Lemma~\ref{f} 
we have $ L_{2n-1}=L_{n}L_{n-1}-L_{n-1}L_{n-2}$ and 
$ A_{2n}=aL_{n}L_{n-1}-aL_{n-1}L_{n-2}-2L_{2n-2}-2 =a^{2}L_{n-1}^{2}-2aL_{n-1}L_{n-2}-2L_{2n-2}-2$.
The following are routine calculations: 
$ L_{2n-2}=L_{n-1}^{2}-L_{n-2}^{2} $; $ A_{2n}=a^{2}L_{n-1}^{2}-2aL_{n-1}L_{n-2}-2L_{n-1}^{2}+2L_{n-2}^{2}-2$;
$A_{2n}=(a^{2}L_{n-1}^{2}-4aL_{n-1}L_{n-2}+4L_{n-2}^{2}-4)-2L_{n-1}^{2}; -2L_{n-1}^{2}+2aL_{n-1}L_{n-2}+2$.
Then by Lemma~\ref{g} $-2L_{n-1}^{2}-2L_{n-1}^{2}+2aL_{n-1}L_{n-2}+2=0$.So $A_{2n}=a^{2}L_{n-1}^{2}-4aL_{n-1}L_{n-2}+4L_{n-2}^{2}-4=
(aL_{n-1}-2L_{n-2})^{2}-4=
(aL_{n-1}-2L_{n-2}-2)(aL_{n-1}-2L_{n-2}+2)=A_{n}(A_{n}+4)$.

Note that $ A_{2}=a^{2}-4 $ is not the determinant of the Laplacian of $ Z_{2} $ and
 $ A_{1}=a-2 $ is not the determinant of the Laplacian of the trivial group $E$. But $ \lambda =0 $ 
is the eigenvalue of all Laplacian because each Laplacian is a singular matrix. 
However, by the Table 2 we see that$ \lambda =2 $ is the eigenvalue of Laplacian with the multiplicity $ 2 $ of 
$ Z_{4k}, k\in N $.

\end{proof}
Note that $ \lambda =2 $ is not  the eigenvalue of the  Laplacian  of $ Z_{4k-2}, k\in N $. It is easy to see that $A_{4k-2}(0)=-2$.
\begin{theorem} \label{d}
 If $ \lambda $ is the eigenvalue of the Laplacian of $ Z_{n}, n\geq 3 $, then $ \lambda $ is 
the eigenvalue of the Laplacian of $ Z_{kn}, \forall k\in N $.
\end{theorem}
\begin{proof} Lemma~\ref{e}  yields that $ A_{n} $ is a divisor of $ A_{2n} $. Now suppose $ A_{n} $ is a divider of $ A_{mn}, \forall m\leq k $.

$ A_{n}=aL_{n-1}-2L_{n-2}-2=L_{n}-L_{n-2}-2$;
$ A_{(k+1)n}=L_{(k+1)n}-L_{(k+1)n-2}-2=L_{kn}L_{n}-L_{kn-1}L_{n-1}-L_{kn}L_{n-2}+L_{kn-1}L_{n-3}-2=L_{kn}(L_{n}-L_{n-2}-2)+2L_{kn}-2+L_{kn-1}L_{n-3}-L_{kn-1}L_{n-1}=A_{n}L_{kn}+2L_{kn}-2+L_{kn-1}L_{n-3}-aL_{kn-2}L_{n-1}+L_{kn-3}L_{n-1}+L_{kn-2}L_{n-2}-L_{kn-2}L_{n-2}=
A_{n}L_{kn}+2L_{kn}-2+L_{kn-1}L_{n-3}+L_{kn-3}L_{n-1}-(L_{kn-2}L_{n}-L_{kn-2}L_{n-2}-2L_{kn-2})-2L_{kn-2}-2L_{kn-2}L_{n-2}=A_{n}L_{kn}-A_{n}L_{kn-2}+2(L_{kn}-L_{kn-2}-2)+2+L_{kn-1}L_{n-3}+L_{kn-3}L_{n-1}-2L_{kn-2}L_{n-2}=A_{n}L_{kn}-A_{n}L_{kn-2}+2A_{kn}+B$.

$ B=2+L_{kn-1}L_{n-3}+L_{kn-3}L_{n-1}-2L_{kn-2}L_{n-2}=
2+L_{kn-2}L_{n-4}+L_{(k+1)n-4}+L_{kn-4}L_{n-2}+L_{(k+1)n-4}-2L_{kn-3}L_{n-3}-
2L_{(k+1)n-4}=\ldots=2+L_{(k-1)n+3}L_{1}+L_{(k-1)n+1}L_{3}-2L_{(k-1)n+2}L_{2}= 2+L_{(k-1)n+2}+L_{(k-1)n+4}+L_{(k-1)n}L_{2}+L_{(k-1)n+4}-2L_{(k-1)n+1}L_{1}-
2L_{(k-1)n+4}=(a^{2}-1)L_{(k-1)n}-aL_{(k-1)n-1}+(a^{2}-1)L_{(k-1)n}-2a^{2}L_{(k-1)n}+2aL_{(k-1)n-1}+2=
-2L_{(k-1)n}+aL_{(k-1)n-1}-L_{(k-1)n-2}+L_{(k-1)n-2}+2=-L_{(k-1)n}+L_{(k-1)n-2}+2=-A_{(k-1)n}$.

So $ A_{(k+1)n}=A_{n}L_{kn}-A_{n}L_{kn-2}+A_{kn}-A_{(k-1)n}=(A_{n}+2)A_{kn}+2A_{n}-A_{(k-1)n} $ and $ A_{n} $ is a divider of $ A_{(k+1)n} $ .

\end{proof}
For example we can prove that the Laplacian spectra of $Z_{2}\times Z_{3}$ and $Z_{6}$ are 
different. The graph of $Z_{2}\times Z_{3}$ is isomorphic to the complement of the graph
of $Z_{6}$. It is well known that if $\lambda\neq 0$ is the eigenvalue of $L(G)$, then $n-\lambda$
is the eigenvalue of $L(G^{C})$, see~\cite{turker}. Since $\lambda=4 $ is the eigenvalue of $Z_{6}$
it follows that $\lambda=2$ is the eigenvalue of $Z_{2}\times Z_{3}$ and $\lambda=2$ is not the eigenvalue 
of $Z_{6}$. Therefore, the spectra of isomorphic groups can be different. Note that $Z_{2}\times Z_{2}\ncong Z_{4}$ but 
their graphs and Laplacian spectra coincide.
\begin{lemma} 
$ A_{kn+p}=(A_{p}+2)A_{kn}+2A_{p}-A_{kn-p} $.
\end{lemma}
\begin{proof}
$ A_{kn+p}=L_{kn+p}-L_{kn+p-2}-2=L_{kn}L_{p}-L_{kn-1}L_{p-1}-L_{kn}L_{np-2}+L_{kn-1}L_{p-3}-2=L_{kn}(L_{p}-L_{p-2}-2)+2L_{kn}-2+L_{kn-1}L_{n-3}-L_{kn-1}L_{p-1}=A_{p}L_{kn}+2L_{kn}-2+L_{kn-1}L_{p-3}-aL_{kn-2}L_{p-1}+L_{kn-3}L_{p-1}+L_{kn-2}L_{p-2}-L_{kn-2}L_{p-2}=
A_{p}L_{kn}+2L_{kn}-2+L_{kn-1}L_{p-3}+L_{kn-3}L_{p-1}-(L_{kn-2}L_{p}-L_{kn-2}L_{p-2}-2L_{kn-2})-2L_{kn-2}-2L_{kn-2}L_{p-2}=A_{p}L_{kn}-A_{p}L_{kn-2}+2(L_{kn}-L_{kn-2}-2)+2+L_{kn-1}L_{p-3}+L_{kn-3}L_{p-1}-2L_{kn-2}L_{p-2}=A_{p}L_{kn}-A_{p}L_{kn-2}+2A_{kn}+B$.

$ B=2+L_{kn-1}L_{p-3}+L_{kn-3}L_{p-1}-2L_{kn-2}L_{p-2}=
2+L_{kn-2}L_{p-4}+L_{(k+1)n-4}+L_{kn-4}L_{p-2}+L_{(k+1)n-4}-2L_{kn-3}L_{p-3}-
2L_{(k+1)n-4}=\ldots=2+L_{(k-1)n+3+(n-p)}
L_{1}+L_{(k-1)n+1+(n-p)}L_{3}-2L_{(k-1)n+2+(n-p)}L_{2}= 2+L_{(k-1)n+2+(n-p)}+L_{(k-1)n+4+(n-p)}+L_{(k-1)n+(n-p)}L_{2}+L_{(k-1)n+4+(n-p)}-2L_{(k-1)n+1+(n-p)}L_{1}-
2L_{(k-1)n+4+(n-p)}=\ldots=-A_{(k-1)n+(n-p)}$.
 
So $ A_{kn+p}=A_{p}L_{kn}-A_{p}L_{kn-2}+2A_{kn}-A_{(k-1)n+(n-p)}=(A_{n}+2)A_{kn}+2A_{p}-A_{kn-p} $.\\
By Theorem~\ref{d} we get $$ A_{n+p}=A_{n}(A_{p}+2)+2A_{p}-A_{n-p},p<n\eqno (4)$$ .
\end{proof}
\begin{theorem} 
If $ \lambda \neq 2$ is the eigenvalue of Laplacian of $ Z_{n} $ and $ Z_{m} $, then $ \lambda $ is the eigenvalue 
of the Laplacian of $ Z_{d}$, where $d$ is the greatest common divisor of $m$ and $n$. Moreover, If 
$ \lambda =4$ is the eigenvalue of the Laplacian of $Z_{n}$, then $\exists k\in N: n=2k$. 
Also, if $ \lambda =2$ is the eigenvalue of the Laplacian of $ Z_{n} $, then 
$\exists k\in N: n=4k$ or $ n=2 $.
\end{theorem}
\begin{proof}
Note that if $ \lambda =4$, then $ A_{2}(2-\lambda)=0 $. Assume that $ \lambda \neq 2$ is the 
eigenvalue of the Laplacian of $ Z_{n} $ and $ Z_{m}$ when $ m>n, m=n+k $, and  that the greatest common divisor 
of $m$ and $n$ is 1. Set $ a=2-\lambda $. Then $ A_{n+k}(a)=A_{n}(a)=0 $. By the (4) we have:
$ A_{2n+k}(a)=A_{n+k}(a)(A_{n}(a)+2)+2A_{n}(a)-A_{k}(a)=-A_{k}(a) $
$ A_{2n+2k}(a)=A_{2n+k}(a)(A_{k}(a)+2)+2A_{2n}(a)-A_{k}(a)=-A_{k}^{2}(a) $.
But $ A_{2n+2k}(a)=A_{2(n+k)}(a)=0 $ by the (4.1). So $ A_{k}(a)=0 $.If $k$ and $n$ have the common divisor $ >1 $ then $m$ and $n$ have the common divisor $ >1 $ too. So the greatest common divisor of $k$ and $min(n,n+k)$ is 1. 
Continuing this procedure for the $k$ and $min(n,n+k)$ we obtain the following:
$$ A_{min(k,min(n,n+k))}(a)=A_{\vert k-min(n,n+k)\vert}(a)=0 $$  
In addition, the greatest common divisor of $min(k,min(n,n+k))$ and $\vert k-min(n,n+k)\vert$ is 1.
Continuing this procedure further we prove for some $p$ that $A_{p}(a)=A_{1}(a)=0$. So if the
greatest common divisor of $m$ and $n$ is 1, then $A_{m}(a)=A_{n}(a)=A_{1}(a)=0\Rightarrow a=2$ and $ \lambda =0$.
\end{proof}
Note then the multiplicity of the first eigenvalue $\lambda =0$ is equal 
to the number of components of graph (see~\cite{dm},\cite{turker}). So for all cyclic groups the multiplicity of 
$\lambda =0$  is 1.
\begin{lemma} \label{c}
\textit{$ A_{n}(a)=aA_{n-1}(a)-A_{n-2}(a)+2A_{1}(a),n\geq 3$} 
\end{lemma}
\begin{proof} $A_{n}=aL_{n-1}-2L_{n-2}-2=a(aL_{n-2}-L_{n-3})-2L_{n-2}-2= a^{2}L_{n-2}-L_{n-3}-2L_{n-2}-2=
a^{2}L_{n-2}-2aL_{n-3}-2a+aL_{n-3}+2a-2L_{n-2}-2=
a(aL_{n-2}-2L_{n-3}-2)+aL_{n-3}+2a-2L_{n-2}-2=
aA_{n-1}+aL_{n-3}+2a-2-2(aL_{n-3}-L_{n-4})=
aA_{n-1}-aL_{n-3}+2L_{n-4}+2+2a-4=
aA_{n-1}-A_{n-2}+2A_{1} $.
\end{proof}
\begin{lemma}\label{a} 
$ A_{kn}=A_{k}\circ (A_{n}+2)$. 
\end{lemma}
\begin{proof}
$ A_{2n}=A_{n}(A_{n}+4)=(A_{n}+2-2)(A_{n}+4)=(A_{n}+2)^{2}-4=A_{2}\circ (A_{n}+2)$. Now assume 
$\forall m\leq k: A_{mn}=A_{m}\circ (A_{n}+2) $.\\
$ A_{(k+1)n}(a)=A_{kn}(a)(A_{n}(a)+2)+2A_{n}(a)-A_{(k-1)n}(a)=
(A_{n}(a)+2)A_{k}\circ (A_{n}(a)+2)-A_{(k-1)n}\circ (A_{n}(a)+2)+2(A_{n}(a)+2)-4=
(A_{n}(a)+2)A_{k}\circ (A_{n}(a)+2)-A_{(k-1)n}\circ (A_{n}(a)+2)+2A_{1}\circ (A_{n}(a)+2)$.\\
Hence by Lemma~\ref{c}, $A_{(k+1)n}(a)=A_{k+1}\circ (A_{n}(a)+2) $.
\end{proof}
\begin{theorem} If $ \lambda $ is the eigenvalue of the Laplacian of $ Z_{n}, n\geq 3 $ with the 
multiplicity $r$, then $ \lambda $ is the eigenvalue of the Laplacian of 
$ Z_{kn}, \forall k\in N $ with the multiplicity $r$. Furthermore, If $ \lambda =4$ is eigenvalue of 
the Laplacian of $ Z_{n} $, then $\exists k\in N: n=2k$ and the multiplicity of $\lambda$ is 1. Also
if $ \lambda =2$ is the eigenvalue of the Laplacian of $ Z_{n},n>2$, then $\exists k\in N: n=4k$ 
and the multiplicity of $\lambda=2$ is $2$.
\end{theorem}
\begin{proof}
Assume that $ \lambda_{0}$ is the eigenvalue of the Laplacian of $ Z_{n} $ with the multiplicity $r$ and that of the Laplacian of $ Z_{kn} $ with the multiplicity $q$.
Put $ a_{0}=2-\lambda_{0}$. Obviously $ r\leq q $. Now suppose $ r<q $. Then by Lemma~\ref{a} we have $ A_{kn}=A_{k}\circ (A_{n}+2)=\prod_{i=1}^{k}(A_{n}+2-a_{i})$, where
$a_{i}$ are the roots of $ A_{k}(a)=0$. Note that $ \exists! a_{i}: A_{n}+2-a_{i}=0$ and  $a_{i}=2 $.
Since $ r\leq q $ then $ A_{kn}/A_{n}=\prod_{j=1}^{k-1}(A_{n}(a)+2-a_{j})=0$, where $a_{j}$ are the roots of
$ A_{k}(a)=0$ and $\forall j: a_{j}\neq 2$. $ A_{n}(a)=0\Rightarrow A_{kn}/A_{n}=\prod_{j=1}^{k-1}(2-a_{j})=0$.
But $\forall j: a_{j}\neq 2$. So $ r\nless q\Rightarrow r=q$.

\end{proof}

\begin{theorem} \label{b}
If $ \lambda_{0}\neq 2 $ is the eigenvalue of the Laplacian of $ Z_{n}$, then 
$\forall m\in N: P_{m}(\lambda_{0})=-A_{m}(2-\lambda_{0})=\lambda_{1},$ where $\lambda_{1}$ is the
 eigenvalue of the Laplacian of $ Z_{n}$.
\end{theorem}
\begin{proof}
By Lemma~\ref{a} we get $ A_{mn}(2-\lambda_{0})=\prod_{j=1}^{n}(A_{m}(2-\lambda_{0})+2-(2-\lambda_{j}))=\prod_{j=1}^{n}(A_{m}(2-\lambda_{0})+\lambda_{j})=0$. 
Thus, $\exists \lambda_{1}: P_{m}(\lambda_{0})=-A_{m}(2-\lambda_{0})=\lambda_{1}$, where $\lambda_{1}$ is the eigenvalue of the Laplacian of $ Z_{n}$.
\end{proof}
\begin{cor} 
If $ \lambda $ is the eigenvalue of Laplacian of $ Z_{n}$, then $\lambda\in [0,4]$.
\end{cor}
\begin{proof} Since all $\lambda\geq 0$, then by Theorem~\ref{b} we have
$\forall \lambda_{0}: P_{2}(\lambda_{0})=\lambda_{0}(4-\lambda_{0})\geq 0$.
\end{proof}
\begin{cor} $P_{k}(\lambda)=\lambda_{i}$, where $ \lambda_{i} $ is the eigenvalue of
the Laplacian of $Z_{n}$ $\Leftrightarrow \lambda$ is the eigenvalue of the Laplacian of $ Z_{kn}$.
\end{cor}
\begin{proof} By Lemma~\ref{a} we see that $P_{kn}=(-1)^{n-1}(P_{k}-\lambda_{j})$, where 
$\lambda_{j}$ are the eigenvalues of the Laplacian of $Z_{n}$.
\end{proof}

\end{document}